\documentclass[12pt]{article}
\usepackage{graphicx} 

\usepackage{tikz}

\usepackage{amsmath,amsthm,amscd,amssymb,eucal,mathrsfs}
\usepackage{amsfonts}
\usepackage{latexsym}
\usepackage{graphicx} 
\usepackage{mathtools}
\usepackage{multicol}
\usepackage{dsfont}
\usepackage[all]{xy}
\usepackage{multirow}
\usepackage{color}
\usepackage{mathtools}
\usepackage[hidelinks]{hyperref}
\usepackage[all]{xy}

\newtheorem{thm}{Theorem}[section]

\newtheorem{remark}[thm]{Remark}

\newtheorem*{Satz*}{Satz}

\newtheorem{Lemma}[thm]{Lemma}

\newtheorem{Corollary}[thm]{Corollary}

\newcommand{\mathset}[1]{{\left\{#1\right\}}}
\newcommand{\absolute}[1]{\left\lvert#1\right\rvert}

\DeclareMathOperator{\Spec}{Spec}

\DeclareMathOperator{\trace}{tr}

\DeclareMathOperator{\ind}{ind}

\DeclareMathOperator{\res}{res}

\DeclareMathOperator{\dist}{dist}

\title{Topological Applications of $p$-Adic Divergence and Gradient Operators}
\author{Patrick Erik Bradley}
\date{\today}

\begin{document}

\maketitle

\begin{abstract}
$p$-Adic divergence and gradient operators are constructed giving rise to $p$-adic vertex Laplacian operators used by Z\'u\~niga in order to study Turing patterns on graphs, as well as their edge Laplacian counterparts. It is shown that the Euler characteristic of a finite graph can be expressed via
traces of certain  heat kernels 
associated with these new operators. This result is applied to the extraction of topological information from Mumford curves via heat kernels.
\end{abstract}

\section{Introduction}

The well-known  relationship between the divergence and gradient operators for functions on $\mathds{R}^n$ and corresponding vector fields reveals them as mutually adjoint linear operators. In equations describing the motion of particles, these two operators are used in order to describe advection and diffusion. The former uses the divergence operator, and the latter the Laplacian which is gradient followed by divergence.
In a discrete (finite) setting, these two operators are mimicked by the incidence matrix (divergence) and its transpose (gradient). The mathematical significance of these is that the singular value theorem says that the two Laplacians which can be constructed from these (edge Laplacian and vertex Laplacian) have a common spectrum whose difference lies only in the multiplicity of eigenvalue zero. In the case of finite graphs, this difference turns out to coincide with the Euler characteristic of the graph. 
Using the  heat kernels corresponding to the two Laplacians then leads to the well-known result that the index
of the gradient operator equals the difference between the traces of these heat kernels. 
Already Roth had seen one ``half'' of this theorem by proving that the trace of the heat kernel associated with the (vertex) Laplacian of a finite graph has an expansion into a term containing the Euler characterstic and decaying exponential functions over the spectrum \cite{Roth1984}.
\newline

It is a desideratum to obtain results towards a $p$-adic formulation of an index theorem for $p$-adic manifolds, possibly using integral operators defined on these. In this context, the present article poses a first contribution for $p$-adic integral operators defined via finite graphs in a similar manner as W.\ Z\'u\~niga-Galindo has done in the study of Turing patterns \cite{ZunigaNetworks}. The difference is that here, the vertices are represented by  compact open subsets of the field $\mathds{Q}_p$ of $p$-adic numbers.
\newline

Using a
$p$-adically defined divergence operator $d^\alpha$ and a corresponding gradient operator $\delta^\alpha$ for $\alpha>0$, which do not seem to have precedents in the literature, corresponding $p$-adic vertex and edge Laplacian operators 
$D^\alpha$ and $\Delta^\alpha$
can be defined which are self-adjoint and positive semi-definite.
The vertex operator $D^\alpha$ then is the generalisation of a Z\'u\~niga operator also to be found in \cite{brad_heatMumf}.
\newline

The main result of this article is given as follows:
\newline

\noindent
{\bf Theorem.} \emph{
If every vertex of a finite graph $G$ has a loop-edge, then
\[
\lim\limits_{t\to\infty}\left(
\trace(p_{D^\alpha,t})-\trace(p_{\Delta_\rho^\alpha,t})\right)
=\chi(G')=\ind(d^\alpha)
\]
for $\alpha>0$.
Otherwise, it holds true that
\[
\lim\limits_{\ell\to\infty}\lim\limits_{t\to\infty}
\left(
\trace(p_{D_\ell^\alpha,t})-\trace(p_{\Delta_\rho^\alpha,t})
\right)
=\chi(G')=\ind(d^\alpha)
\]
for $\alpha>0$. The graph $G'$ is obtained from $G$ by removing all loop-edges.
}
\newline

The expression 
$p_{D^\alpha,t}$ denotes the heat kernel associated with $D^\alpha$, whereas
$p_{D_\ell^\alpha,t}$ and $p_{\Delta_\rho^\alpha}$ are heat kernels associated with finite-dimensional  operators derived from $D^\alpha$ and $\Delta^\alpha$ where cut-offs are made through selecting eigenfunctions. The reason for the need to take limits is that on the one hand, the kernel of $\Delta^\alpha$ is infinite-dimensional, and the heat kernel of $D^\alpha$ does not have a well-defined trace unless all vertices of the graph $G$ have loop-edges attached to them. Furthermore, the index of $d^\alpha$, in order to be finite, is defined using  the kernel of the operator $\Delta_\rho^\alpha$ instead of that of $\Delta^\alpha$, but is using the kernel of $D^\alpha$.
\newline

Finally, the theorem is applied to Mumford curves 
by constructing three kinds of graphs from a good fundamental domain in the sense of \cite[Def.\ I.4.1.3]{GvP1980}. In this way, topological information can be extracted from a Mumford curve by  heat kernels. This continues in a sense the hearing of shapes for graphs or $p$-adic geometric objects via $p$-adic analysis, as done in \cite{BL_shapes_p,Brad_HearingGenusMumf}.
For applications, this could be of interest, since on the one hand, using hierarchies  allows a faster processing of structured data, and on the other hand using integral operators to extract topological information from data is  a main topic of applied research.
\newline

The goal of the ongoing project \emph{Distributed Simulation of 
Processes on Buildings and City Models} (project number 469999674)
aims to provide a framework for simulating processes like e.g.\ heat flows on building and city models on distributed systems via $p$-adic numbers. 
The building or city model topologies can be represented as a $T_0$-topology on  the finite set of atomic objects of the model connected by a binary relation which can be interpreted as a boundary relationship \cite{BradPaul2010}. The Hasse diagram is then a minimal representation of this $T_0$-topology.  For details on the properties of such topological spaces, cf.\ \cite{Alexandrov1937}.
\newline

The choice of $p$-adic numbers as an underlying structure for the DFG project
seems natural because of  their inherent hierarchical structure, and the fact that processing hierarchical information in general is efficient due to their tree-like structuring. Since $p$-adic analysis is relatively young, the mathematics behind processes on hierarchically structured data still needs to be developped in many cases.
The concern of a large part of $p$-adic analysis is diffusion, whereas other types of processes, e.g.\ advection, have not yet been studied to the author's knowledge. 
The gradient and divergence operators developped here can fill a part of this gap by providing a $p$-adic advection operator:
\[
D_+^\alpha=\delta_+^\alpha d_+^\alpha\colon L^2(\Omega_{V})\to L^2(\Omega_{V})
\]
for $\alpha>0$, using the notation of Section 3.
This is in fact a natural generalisation of advection operators built from the outgoing incidence matrix of finite graphs introduced in \cite{CM2011}. The study of $p$-adic advection equations like
\[
\frac{\partial}{\partial t}u(x,t)+D_+^\alpha u(x,t)=0
\]
is left for future work in the ongoing project.
\newline

Another significance of the results of this article for the DFG-project is that the task of retrieving topological information from $p$-adic diffusion processes becomes more feasible. A further significance is that Mumford curves as compact $p$-adic manifolds are gradually coming   into the focus of data analysis, 
not only because of the results of the previous section, but also because of recent project results showing that a hierarchical dataset can be viewed as being sampled from a Mumford curve together with a regular differential $1$-form on it. This leads to
the question about the possible shapes such a Mumford curve can take, including the question about its genus.
\newline

The following section lists some preliminaries helpful for reading this article.
Section 3 constructs the divergence and gradient operators, shows that they are mutually adjoint, thus giving rise to two kinds of $p$-adic Laplacians, and 
shows how to obtain the Betti numbers of a finite graph from these. Section 4 constructs the heat kernels, using the result from \cite{ZunigaNetworks} that the $p$-adic vertex Laplacian defines a strong Markov process, and proves the main theorem.
Section 5 applies the main theorem to Mumford curves in order to extract topological information via heat kernels. 

\section{Preliminaries}

Some familiarity with $p$-adic numbers is assumed. The field of $p$-adic numbers is denoted as $\mathds{Q}_p$, and the $p$-adic absolute value with $\absolute{\cdot}_p$. As a locally compact abelian group, $\mathds{Q}_p$ is endowed with a Haar measure, denoted as $dx$ or $dy$, depending on the variable of integration used. It is assumed normalised such that 
\[
\int_{\mathds{Z}_p}dx=1
\]
where $\mathds{Z}_p\subset \mathds{Q}_p$ is the $p$-adic unit disc, and where integration means that of real- or complex-valued functions on $\mathds{Q}_p$ against the Haar measure.
\newline

From $p$-adic analysis, variants of  the integral operator version of the Vladimirov operator are used. The latter is studied e.g.\ in \cite{VVZ1994}. The variants of interest here are first defined in \cite{ZunigaNetworks}. Of use for determining the spectra of such operators are the Kozyrev wavelets, where such is a function of the form
\[
\psi(x)=p^{-\frac{d}{2}}
\chi(p^{d-1}jx)\Omega\left(\absolute{p^{d}x-n}_p\right)
\]
for $d\in\mathds{Z}$, $j\in\mathset{1,\dots,p-1}$, and $n\in\mathds{Q}_p$ the fractional part of a representative of an element of $\mathds{Q}_p/\mathds{Z}_p$, cf.\ \cite{Kozyrev2002}.
By the fractional part of a $p$-adic number
\[
x=\sum\limits_{k=-m}\alpha_kp^k
\]
with $m\in\mathds{Z}$ $\alpha_k\in\mathset{0,d\dots,p-1}$
is meant the expression
\[
\{x\}_p=\sum\limits_{k=-m}^{-1}\alpha_kp^k
\]
as a rational number.
The indicator function $\Omega(\absolute{f(x)}_p)$ is defined as
\[
\Omega(\absolute{f(x)}_p)=\begin{cases}
1,&\absolute{f(x)}_p\le 1
\\
0,&\text{otherwise}
\end{cases}
\]
for a given function $f\colon \mathds{Q}_p\to\mathds{Q}_p$. 
\newline

Some theory of finite graphs, viewed as $1$-dimensional simplicial complexes is also helpful. A lot more than needed for this article about the homology and cohomology of such objects can be found in e.g.\ \cite{HatcherAT}.
\newline

The last section deals with applications to Mumford curves. These are treated in depth in \cite{GvP1980,FP2004}. What is needed here, however, is only their relationship to finite simple graphs in order to see that the $p$-adic version of an index theorem can be applied to these.

\section{$p$-Adic Divergence and Gradient Operators}

Let $G=(V,E)$ be a finite undirected  graph with vertex set $V$ and edge set $E$.
The graph $G$ is assumed to have at most one loop-edge attached to a vertex.
By a loop-edge is meant an edge connecting a vertex with itself.
In particular, it is 
simple after removing all loop-edges. 
\newline 

The subset of loop-edges is denoted as $E_0$, and 
\[
E'=E\setminus E_0
\]
is the set of all simple edges of $G$.
By choosing for each edge from $E'$ a positive orientation, obtain the sets $E_+,E_-$ of positively and negatively oriented simple edges of $G$.
This leads to one way of viewing simple graphs as $1$-dimensional simplicial complexes which can be learned e.g.\ in \cite[Ch.\ 2]{HatcherAT} and elsewhere.
\newline

Generalising an idea of Wilson Z\'u\~niga-Galindo, identify
the vertices $v\in V$ with pairwise disjoint $p$-adic compact open sets $U_v\subset\mathds{Q}_p$ satisfying the condition
\begin{align}\label{setCondition}
\forall v,w\in V\colon v\neq w\;\Rightarrow\;
\forall x\in U_v\,\forall y\in U_w\colon\absolute{x-y}_p=\dist(U_v,U_w)
\end{align}
where $\dist$ is the $p$-adic distance between sets, defined as
\[
\dist(A,B)=\inf\mathset{\absolute{x-y}_p\mid x\in A,\;y\in B}
\]
for subsets $A,B\subseteq\mathds{Q}_p$, and 
define now the compact open sets:
\begin{align*}
\Omega_V&=\bigsqcup\limits_{v\in V}U_v\subset \mathds{Q}_p
\\
\Omega_{E_0}&=\bigsqcup\limits_{e\in E_0 }U_{o(e)}\times U_{o(e)}\subset\mathds{Q}_p^2
\\
\Omega_{E_+}&=\bigsqcup\limits_{e\in E_+}
U_{o(e)}\times U_{t(e)}
\subset\mathds{Q}_p^2
\\
\Omega_{E_-}&=\bigsqcup\limits_{e\in E_+}
U_{t(e)}\times U_{o(e)}\subset\mathds{Q}_p^2
\\
\Omega_{E'}&=\Omega_{E_+}\cup\Omega_{E_-}\subset\mathds{Q}_2^2
\\
\Omega_E&=\Omega_{E_0}\cup\Omega_{E_+}\cup\Omega_{E_-}\subset\mathds{Q}_p^2
\end{align*}
where $o(e)\in V$ denotes the origin vertex of an oriented edge $e\in E_+$, and $t(e)$ its terminal vertex.
Note that if $e\in E_0$, then $o(e)=t(e)$.
\newline

Let $\alpha>0$, and  define the following $p$-adic divergence operators:
\[
\delta_\bullet^\alpha\colon L^2(\Omega_{E_\bullet})\to L^2(\Omega_V),\;h\mapsto \delta_\bullet^\alpha h
\]
with 
\[
(\delta_\bullet^\alpha h)(x)=
\int_{\Omega_{V}}h(x,y)\absolute{x-y}_p^{-\alpha}\,dy
\]
for $h\in L^2(\Omega_{E_\bullet})$ and $\bullet\in\mathset{+,-,0}$.
Define also the following $p$-adic gradient operators:
\[
d_0^\alpha\colon L^2(\Omega_V)\to L^2(\Omega_{E_0}),\;f\mapsto d_0^\alpha f
\]
with
\[
(d_0^\alpha f)(x,y)
=\frac{f(x)-f(y)}{\absolute{x-y}_p^\alpha}
\]
for $(x,y)\in \Omega_{E_0}$,
\[
d_+^\alpha\colon L^2(\Omega_V)\to L^2(\Omega_{E_+}),\;
f\mapsto d_+^\alpha f
\]
with 
\[
(d_+^\alpha f)(x,y)=\frac{f(x)-f(y)}{2\absolute{x-y}_p^\alpha}
\]
for $(x,y)\in \Omega_{E_+}$, and
\[
d_-^\alpha\colon L^2(\Omega_V)\to L^2(\Omega_{E_+})
\]
with
\[
(d_-f)(x,y)=(d_+f)(y,x)
\]
for $(x,y)\in \Omega_{E_-}$.
\newline

The inner products on $L^2(\Omega_{E_\bullet}),L^2(\Omega_V)$ will be denoted as
\[
\langle\cdot,\cdot\rangle_{E_\bullet},\quad\langle\cdot,\cdot\rangle_V
\]
and are defined in a standard way via integration.
Notice that the operators $\delta_0^\alpha,d_0^\alpha$ are each unbounded operators on their respective Hilbert spaces.
\newline

Define the spaces
\begin{align*}
L^2(\Omega_E)_\sigma&:=\mathset{h\in L^2(\Omega_E)\mid \forall (x,y)\in \Omega_E\colon h(x,y)=-h(y,x)}
\\
L^2(\Omega_{E_0})_\sigma&:= L^2(\Omega_E)_\sigma\cap L^2(\Omega_{E_0})
\end{align*}
with inner product for the larger space denoted as $\langle\cdot,\cdot\rangle_E$. Any function $h\in L^2(\Omega_E)_\sigma$ can be uniqueley decomposed as
\[
h=h_0+h_+-h_-
\]
with $h_0\in L^2(\Omega_{E_0})_\sigma$, $h_+\in L^2(\Omega_{E_+})$ and $h_-\in L^2(\Omega_{E_-})$.
Define now the operator
\[
\delta^\alpha\colon
L^2(\Omega_E)_\sigma\to L^2(\Omega_V)
\]
as
\[
\delta^\alpha=\delta_0^\alpha+
\delta^\alpha_+ +\delta_-^\alpha
\]
and finally the operator
\[
d^\alpha\colon L^2(\Omega_V)\to L^2(\Omega_E)_\sigma
\]
as
\[
d^\alpha=d_0^\alpha+d_+^\alpha+d_-^\alpha
\]
for $\alpha>0$.

\begin{Lemma}\label{AdvectionPair}
It holds true that
\[
\langle d^\alpha f,h
\rangle_{E}
=\langle f,\delta^{\alpha} h\rangle_V
\]
for $h\in L^2(\Omega_E)_\sigma$, $f\in L^2(\Omega_V)$.
In other words, $(d^\alpha,\delta^{\alpha})$ is a pair of adjoint operators.
\end{Lemma}

\begin{proof}
Let $f\in L^2(\Omega_V)$ and $h\in L^2(\Omega_{E})_\sigma$. Observe that
\[
h_+(x,y)=h(x,y),\quad h_-(y,x)=-h(x,y)
\]
for $(x,y)\in \Omega_{E_+}$. 
It follows that
\begin{align*}
\langle d^\alpha f,h_++h_-\rangle_{E}
&=\int_{\Omega_{E_+}}
\frac{f(x)-f(y)}{2\absolute{x-y}_p^{\alpha}}\overline{h(x,y)}\,dx\,dy
-
\int_{\Omega_{E_+}}
\frac{f(y)-f(x)}{2\absolute{x-y}_p^{\alpha}}
\overline{h(x,y)}\,dx\,dy
\\
&=\int_{\Omega_{E_+}}
\frac{f(x)-f(y)}{\absolute{x-y}_p^\alpha}\overline{h(x,y)}\,dx\,dy
\end{align*}
on the one hand.
On the other hand, \begin{align*}
\langle f,\delta_+^{\alpha} h_+\rangle_V
&
=\int_{\Omega_V}\int_{\Omega_V}
f(x)\overline{h_+(x,y)}
\absolute{x-y}_p^{-\alpha}\,dx\,dy
\\
&=\int_{\Omega_{E_+}}f(x)\overline{h_+(x,y)}\absolute{x-y}_p^{-\alpha}\,dx\,dy
\end{align*}
and
\begin{align*}
\langle f,\delta_-^{\alpha}h_-\rangle_V&=
\int_{\Omega_V}\int_{\Omega_V}
f(x)\overline{h_-(x,y)}\absolute{x-y}_p^{-\alpha}\,dx
\,dy
\\
&=\int_{\Omega_{E_-}}f(x)\overline{h_-(x,y)}\absolute{x-y}_p^{-\alpha}\,dx\,dy
\\
&=-\int_{\Omega_{E_+}}
f(x)
\overline{h_+(y,x)}\absolute{x-y}_p^{-\alpha}\,dy\,dx
\\
&=-\int_{\Omega_{E_+}}
f(y)\overline{h_+(x,y)}\absolute{x-y}_p^{-\alpha}\,dx\,dy
\end{align*}
which takes care of the part concerning the simple edges of $G$. For the loop-edges, assume that $f\in \mathcal{D}(\Omega_V)$, and $h_0\in\mathcal{D}(\Omega_{E_0})_\sigma$, where $\mathcal{D}(Z)$ means the space of locally constant functions on a space $Z$ with compact support, and the subscript $\sigma$ again that $h_0(x,y)=-h_0(y,x)$ for $(x,y)\in\Omega_{E_0}$.
Then it is readily seen that $d^\alpha_0 f\in\mathcal{D}(\Omega_{E_0})_\sigma$ and $\delta_0^\alpha h_0\in\mathcal{D}(V)$. Furthermore, the supports of $h_0$ and $d^\alpha_0 f$ are away from a neighbourhood of the diagonal. Hence, they can be written as a disjoint union of products of two disjoint discs, where both sides each run through  a finite set of discs in $\Omega_V$.
This allows the  decompositions:
\[
h_0=h_{0,+}-h_{0,-},\quad
d_0^\alpha f=(d_0^\alpha f)_+-(d_0^\alpha f)_-
\]
where each part is a  function which is positive, locally constant (and compactly supported).
Then a similar calculation as in the simple-edge case yields that
\begin{align*}
\langle d_0^\alpha f,h_0\rangle_{E_0}
=\langle f,\delta_0^\alpha\rangle_V
\end{align*}
for test functions. By continuity of the inner products, this identity extends to the (restricted) $L^2$-spaces.
Hence, putting things together, arrive at
\[
\langle f,\delta^{\alpha}h\rangle_{V}=\langle f,\delta_0^\alpha h_0\rangle_V+
\langle f,\delta_+^{\alpha}h_+\rangle_V-\langle f,\delta_-^{\alpha}h_-\rangle_V
=\langle d^\alpha f,h\rangle_{E}
\]
as asserted.
\end{proof}

Define the following Laplacian operators:
\begin{align*}
D^\alpha&=\delta^{\alpha}d^\alpha\colon L^2(\Omega_V)\to L^2(\Omega_V)
&\text{($p$-adic vertex Laplacian)}
\\
\Delta^\alpha&=d^\alpha\delta^{\alpha}\colon L^2(\Omega_E)_\sigma\to L^2(\Omega_E)_\sigma
&\text{($p$-adic edge Laplacian)}
\end{align*}
for a finite  simple graph $G=(V,E)$.
Explicitly, $D^\alpha$ is given as follows:

\begin{Corollary}\label{Laplacians}
It holds true that
\[
D^\alpha=\delta_0^\alpha d_0^\alpha+
\delta_+^\alpha d^\alpha_++\delta_-^\alpha d_-^\alpha\quad\text{and}\quad
\Delta^\alpha=d_0^\alpha\delta_0^\alpha+
d_+^\alpha\delta_+^\alpha+d_-^\alpha\delta_-^\alpha
\]
are self-adjoint, positive semi-definite,
and
\[
D^\alpha f(x)=
\int_{\mathset{y\in\Omega_V\mid (x,y)\in \Omega_E}}\frac{f(x)-f(y)}{\absolute{x-y}^{2\alpha}}\,dy
\]
for $f\in L^2(\Omega_V)$ and $\alpha>0$.
In the case that $G$ is a simple graph, i.e.\ $E_0=\emptyset$, the operators $D^\alpha$ and $\Delta_\alpha$ are bounded, otherwise these operators are unbounded.
\end{Corollary}

\begin{proof}
The self-adjointness and positive semi-definitiveness statements are an immediate consequence of Lemma \ref{AdvectionPair}. The decompositions are seen thus:
\[
D^\alpha =\delta_0^\alpha d_0^\alpha+
\delta_+^\alpha d_+^\alpha +\delta_-^\alpha d_-^\alpha
+\delta_0^\alpha d_+^\alpha
+\delta_0^\alpha d_-^\alpha
+\delta_+^\alpha d_-^\alpha
+
\delta_+ d_0^\alpha
+\delta_-^\alpha d_+^\alpha
+\delta_-^\alpha d_0^\alpha
\]
and all the mixed terms  vanish, as their non-vanishing loci are complementary. This proves the first decomposition, and the second one follows in a similar manner. The  assertion about the integral form of operator $D^\alpha$ is an immediate consequence of the decomposition of $D^\alpha$. The last assertion is also immediate.
\end{proof}

\begin{remark}
Notice that according to Corollary \ref{Laplacians}, the vertex Laplacian is the negative of a Vladimirov-type integral Laplacian operator obtained in the case of a complete  graph
by restricting 
the domain and target of the well-known Vladimirov operator, cf.\ e.g.\ \cite{VVZ1994},
to functions supported on $\Omega_V$.
\end{remark}

The following results are about realising the Betti numbers of a finite graph, viewed as a $1$-dimensional simplicial complex. Consult e.g.\ \cite[Ch. 2.1]{HatcherAT} 
for simplicial homology.

\begin{Lemma}\label{Betti0}
It holds true that
\[
\dim\ker(d^\alpha)=\dim\ker(D^\alpha)=b_0(G)
\]
for $\alpha>0$, where $b_0(G)$ is the zero-th Betti number of $G$.
\end{Lemma}

\begin{proof}
The operator $d^\alpha$ clearly vanishes on all functions $f\colon \Omega_V\to\mathds{C}$ constant on the compact open sets $\Omega_C\subseteq\Omega_V$ defined by the vertices of the
connected components $C$ of $G$. 
Now assume that 
 $f\in\ker d^\alpha$.
 First, observe that $f$ is constant on each compact open $U_v$ with $v\in V$. Otherwise, there is an edge $(v,w)$ such that $f$ restricted to $U_w$ misses one of the different values on $U_v$.
 Hence, $f$ takes the same value on sets
 $U_w$ where $w$ is connected to $v$ by an edge in $E$. It now follows that $f$ is constant on $\Omega_C$, where $C$ is a connected component of $G$.
This proves that
\[
\dim\ker d^\alpha=b_0(G)
\]
for $\alpha>0$. The first equality is a consequence of the fact that $\delta^\alpha$ and $d^\alpha$ are each other's adjoint operator.
\end{proof}

The $p$-adic edge Laplacian $\Delta^\alpha$ does not have a finite-dimensional kernel, as can be seen thus:
\begin{align*}
(\delta^\alpha h)(x)&=
\int_{\Omega_V}h(x,y)\absolute{x-y}_p^{-\alpha}\,dy
\\
&=\int_{\mathset{y\in \Omega_V\colon (x,y)\in \Omega_E}}
h(x,y)\absolute{x-y}_p^{-\alpha}\,dy
\\
&=\sum\limits_{w\sim v(x)}
\int_{U_w}h(x,y)\absolute{x-y}_p^{-\alpha}\,dy
\\
&=\sum\limits_{w\sim v(x)}
\dist(x,U_w)^{-\alpha}\int_{U_w}h(x,y)\,dy
\end{align*}
where $v\sim w$ means that vertices $v,w\in V$ are connected by an edge, and $v(x)$ means the unique vertex $v\in V$ for which $x\in U_x$.
The space of functions $h\in L^2(\Omega_E)_\sigma$ such that
each integral in the last sum vanishes, is infinite-dimensional.
Define
\[
L^2(\Omega_E)_\rho
=\mathset{h\in L^2(\Omega_E)_\sigma\mid 
\forall (v,w)\in E\colon
h|_{U_v\times U_{w}}
=\text{const.}}
\]
which is a finite-dimensional vector space.

\begin{Lemma}\label{dimensionRadialEdge}
It holds true that
\[
\dim L^2(\Omega_E)_\rho=\absolute{E_+}
\]
for any finite graph $G$ being simple after removing the loop-edges.
\end{Lemma}

\begin{proof}
In the case that $G$ is itself a simple graph, this is immediate.
Otherwise, observe that a function $h\in L^2(\Omega_E)_\rho$ satisfies
\[
h|_{U_v\times U_v}=0
\]
for all $v\in V$. This now implies the assertion also in the more general case.
\end{proof}

The orthogonal complement of $L^2(\Omega_E)_\rho$ in $L^2(\Omega_E)_\sigma$ is
\[
L^2(\Omega_E)_a
=\mathset{h\in L^2(\Omega_E)_\sigma\mid
\forall(v,w)\in E\colon
\int_{U_v\times U_w}h(x,y)\,dx\,dy=0}
\]
and the following holds true:
\begin{Lemma}
It holds true that
\[
L^2(\Omega_E)_a\subseteq\ker\delta^\alpha
\]
for $\alpha>0$.
\end{Lemma}

\begin{proof}
This is immediate.
\end{proof}

Define
\[
\delta_\rho^\alpha\colon L^2(\Omega_E)_\rho\to L^2(\Omega_V)
\]
as the restriction of
$\delta^\alpha$ to $L^2(\Omega_E)_\rho$. Due to the last result, it fits into 
the commutative diagram:
\[
\xymatrix{
L^2(\Omega_E)_\sigma\ar[dr]^{\delta^\alpha}\ar[d]_{\kappa}&\\
L^2(\Omega_E)_\rho\ar[r]_{\delta_\rho^\alpha}&L^2(\Omega_V)
}
\]
where $\kappa$ is the canonical orthogonal projection.
There is now the following operator:
\[
\Delta_\rho^\alpha=d^\alpha\delta_\rho^\alpha
\colon L^2(\Omega_E)_\rho\to L^2(\Omega_E)_\rho
\]
for $\alpha>0$, which clearly is well-defined, cf.\ the proof of Lemma \ref{dimensionRadialEdge}.
\newline

Define $G'=(V,E')$ to be the simple graph obtained from $G$ by removing all loop-edges.

\begin{Lemma}\label{Betti1}
It holds true that
\[
\ker\delta_\rho^\alpha
=\ker \Delta_\rho^\alpha
=b_1(G')
\]
 for $\alpha>0$, where $b_1(G')$ is the first Betti number of $G'$.
\end{Lemma}

\begin{proof}
In the proof of Lemma \ref{Betti0}, it was shown that the kernel of $d^\alpha$ coincides with the kernel of the co-boundary map $d_G^\alpha\colon \mathds{C}^V\to\mathds{C}^{E}$
of a suitably weighted version of the graph $G$ (using positive powers of $p$-adic distances between the corresponding vertex discs as weights).
The image of the co-boundary map in fact lies in $\mathds{C}^{E'}$,  and $\delta_\rho^\alpha$ is clearly the adjoint of that map. This implies the first equality. The second now follows from graph theory. 
\end{proof}

This last Lemma allows to define the index of $d^\alpha$:
\[
\ind(d^\alpha):=\dim\ker d^\alpha-\dim\ker\delta_\rho^\alpha
\]
which is a natural number for $\alpha>0$.
Using the Euler characterstic $\chi(G')$, defined as
\[
\chi(G')=\absolute{V}-\absolute{E_+}
\]
of the simple graph $G'$, obtain:

\begin{Corollary}\label{ind=chi}
It holds true that
\[
\ind(d^\alpha)=\dim\ker D^\alpha-\dim\ker\Delta_\rho^\alpha
=\chi(G')
\]
for $\alpha>0$.
\end{Corollary}

\begin{proof}
The first equality follows immediately from Lemmas \ref{Betti0} and \ref{Betti1}, using 
the equality
\[
\chi(G')=b_0(G')-b_1(G')
\]
from graph theory, cf.\ e.g.\ \cite[Thm.\ 2.44]{HatcherAT}, and the fact that $b_0(G)=b_0(G')$ holds true.
\end{proof}

\section{Heat kernels}

W.\ Z\'u\~niga-Galindo showed that his $p$-adic graph operators are bounded self-adjoint linear Laplacian operators defining a strong Markov process whose transition function is given for each $x\in\Omega_V$
by a probability measure $p_t(x,\cdot)$ on the Borel
$\sigma$-algebra of $\Omega_V$, and this proves the uniqueness and existence of a solution to the Cauchy problem for the heat equation attached to the operator $D^\alpha$
\cite[Thm.\ 4.2]{ZunigaNetworks}.
Later on, it will be shown that the probability measure $p_t(x,\cdot)$ has a well-defined distribution function, if every vertex of $G$  has a loop-edge attached to it.
\newline

As for the edge Laplacian $\Delta_\rho^\alpha$, which is a finite-dimensional linear operator, the following holds true:

\begin{Lemma}
The heat kernel for $\Delta^\alpha_\rho$ exists and has the form
\[
p_{\Delta^\alpha_\rho,t}(x,y)
=\sum\limits_{\Spec(\Delta_\rho^\alpha)}
e^{-\lambda t}
\phi_\lambda(x)\phi_\lambda(y)
\]
for $\alpha>0$.
\end{Lemma}

\begin{proof}
This is well-known from the theory of finite graphs, cf.\ e.g.\ \cite{Chung1996}. 
\end{proof}

\begin{Lemma}\label{compactnessCriterion}
In the case that every vertex of $G$ has a loop-edge, the operator $D^\alpha$ is compact. In all other cases, $D^\alpha$ has an eigenvalue with infinite multiplicity.
\end{Lemma}

\begin{proof}
In
\cite[Lem.\ 5]{brad_heatMumf}, it was observed that in the case of a simple graph, the operator $D^\alpha$ is a Z\'u\~niga operator. These are shown to have have 
Kozyrev wavelets as eigenfunctions \cite[Thm.\ 10.1]{ZunigaNetworks}. 
The  eigenvalues corresponding to a Kozyrev wavelet have infinite multiplicity for the following reason: let $\psi_B$ be a Kozyrev eigenfunction supported on $B\subset\Omega_V$.
Then
\[
D^\alpha\psi_B(x)
=\int_{\mathset{y\in\Omega_V\mid (x,y)\in \Omega_E}}\absolute{x-y}_p^{-2\alpha}\,dy\,\psi_B(x)
\]
where 
\[
\lambda_B=\int_{\mathset{y\in\Omega_V\mid(x,y)\in \Omega_E}}\absolute{x-y}_p^{-2\alpha}\,dy
\]
is a well-defined positive number, and this is an eigenvalue of $D^\alpha$. If the support $C$ of another Kozyrev wavelet $\psi_C$ is contained in $B$, then the same computation yields
\[
D^\alpha\psi_C(x)=\int_{\mathset{y\in\Omega_V\mid(x,y)\in\Omega_E}}
\absolute{x-y}_p^{-2\alpha}\,dy\,\psi_C(x)
\]
and $\lambda_C=\lambda_B$ holds true. Since the disc $C$ can be arbitrarily small, the mutliplicity of $\lambda=\lambda_B$ is infinite. Notice that this argument makes use only of the fact that $(x,y)\in\Omega_E$ implies that $y\notin U_{v(x)}$. 

\smallskip
Now assume that  vertex $v(x)$ has a loop-edge attached to it. Observe that the operator $D^\alpha$ is in fact of the form of Laplacian operators studied in \cite{Kozyrev2004}.
Its kernel function $T(x,y)$ satisfies condition the important condition (6) of that article which says that 
\[
\absolute{x-y}_p=\text{const.}\quad\Rightarrow\quad T(x,y)=\text{const.}
\]
and in that case, \cite[Thm.\ 3]{Kozyrev2004} shows that the Kozyrev wavelets $\psi_B$ supported in $B\subset \Omega_V$ are eigenfunctions, and also provides an explicit expression for the eigenvalues attached to Kozyrev wavelets. In the case of a vertex $v$ having loop-edge attached to it, and $B\subset U_v$, it takes the form:
\begin{align*}
\lambda_B=\lambda_{d,n}
&=\int_{\mathset{y\in \Omega_V\mid (x,y)\in\Omega_E}\setminus B}\absolute{p^{-d}n-y}_p^{-2\alpha}\,dy
+p^{d(1+2\alpha)}
\\
&=\int_{\mathset{y\in\Omega_V\mid(x,y)\in\Omega_E}\setminus B}\dist(B,y)^{-2\alpha}\,dy+p^{d(1+2\alpha)}
\end{align*}
where
\[
B=\mathset{y\in\mathds{Q}_p\mid\absolute{p^{d}x-n}_p\le 1}
\]
with $d\in\mathds{Z}$, and $n\in\mathds{Q}_p$
is the fractional part of a representative of an element of
$\mathds{Q}_p/\mathds{Z}_p$. 
This can be compared with the case of a simple graph as follows: it holds true that
\[
\lambda_B=\lambda_B^\sigma
+p^{d(1+2\alpha)}
+\int_{U_v\setminus B}\dist(B,y)^{-2\alpha}\,dy
\]
where
\[
\lambda_B^\sigma
=\int_{\mathset{y\in \Omega_V\mid(x,y)\in\Omega_E}\setminus U_v}
\absolute{x-y}_p^{-2\alpha}\,dy
=\int_{\mathset{y\in \Omega_V\mid(x,y)\in\Omega_E}\setminus U_v}
\dist(B,y)^{-2\alpha}\,dy
\]
is the eigenvalue if $v$ has no loop-edge attached. Since $\lambda_B^\sigma=\lambda_C^\sigma$ if disc $C\subseteq B$, it now follows that
$\lambda_C<\lambda_B$ if $C$ is strictly contained in $B$.
This shows that the multiplicity of $\lambda_B$ is finite.
It now follows that $D^\alpha$ is a compact operator, if all vertices have a loop-edge attached, and in all other cases  some eigenvalue has infinite multiplicity, as asserted. 
\end{proof}

\begin{remark}
It
has already been observed that Z\'u\~niga operators are in general not compact \cite[Ex.\ 1]{brad_heatMumf}.
\end{remark}

\begin{Corollary}
Assume that each vertex of $G$ has a loop-edge attached to it.
Then the heat kernel for $D^\alpha$ exists and has the form
\[
p_{D^\alpha,t}(x,y)=\sum\limits_{\lambda\in\Spec(D^\alpha)}
e^{-\lambda t}
f_\lambda(x)\overline{f_\lambda(y)}
\]
where $f_\lambda\in L^2(\Omega_V)$ is a normalised eigenfunction of $D^\alpha$ for eigenvalue $\lambda\in\Spec(D^\alpha)$, 
for $t>0$ and $\alpha>0$.
\end{Corollary}

\begin{proof}
Taking as initial condition
\[
h_0(x)=f_\lambda(x)
\]
obtain from the Cauchy problem:
\[
h(t,x)=\int_{\Omega_V}h_0(y)\,p_{D^\alpha,t}(x,dy)
=e^{-\lambda t}f_\lambda(x)
\]
for $t>0$ and $\alpha>0$, where $p_{D^\alpha,t}(x,\cdot)$ is the probability measure giving rise to a strong Markov process according to \cite[Thm.\ 4.2]{ZunigaNetworks}.
Now, the proposed density function integrates the initial condition in this way:
\begin{align*}
\int_{\Omega_V}f_\lambda(y)&\,p_{D^\alpha,t}(x,y)\,dy
=\sum\limits_{\lambda'\in\Spec(D^\alpha)}e^{-\lambda' t}\int_{\Omega_V}f_\lambda(y)f_{\lambda'}(x)\overline{f_{\lambda'}(y)}\,dy
\\
&=e^{\lambda t}f_\lambda(x)
\end{align*}
giving rise to the same solution of the Cauchy problem. Hence, this happens for any initial condition in $L^2(\Omega_V)$, and the proposed function is the distribution function of $p_{D^\alpha,t}(x,\cdot)$ for any $x\in\Omega_V$. In other words, the heat kernel for $D^\alpha$ exists and has the asserted form.
\end{proof}

In order to formulate a $p$-adic version of an index theorem, two cases need to be considered: either the case where each vertex of $G$ has a loop-edge, or some vertex does not. 
In the latter case,
consider the finite approximation operators $D^\alpha_\ell$ defined as:
\[
D_\ell^\alpha\colon L^2(\Omega_{V,\ell})\to L^2(\Omega_{V,\ell})
\]
with 
\[
D^\alpha_\ell f(x)=
\int_{\Omega_{V,\ell}}\frac{f(x)-f(y)}{\absolute{x-y}^{2\alpha}}\,d_\ell y
\]
where
\[
\Omega_{V,\ell}=\Omega_V\mod p^\ell\mathds{Z}_p
\]
is obtained by from $\Omega_V$ by partitioning into a disjoint union of discs of radius $p^{-\ell}$ for $\ell\in\mathds{N}$, and where
$d_\ell y$ is the counting measure on the finite set $\Omega_{V,\ell}$. Notice that the expression 
\[
\absolute{x-y}_p
\]
in the integral is well-defined.

\begin{thm}\label{IndexTheorem}
If every vertex of $G$ has a loop-edge, then
\[
\lim\limits_{t\to\infty}\left(
\trace(p_{D^\alpha,t})-\trace(p_{\Delta_\rho^\alpha,t})\right)
=\chi(G')=\ind(d^\alpha)
\]
for $\alpha>0$.
Otherwise, it holds true that
\[
\lim\limits_{\ell\to\infty}\lim\limits_{t\to\infty}
\left(
\trace(p_{D_\ell^\alpha,t})-\trace(p_{\Delta_\rho^\alpha,t})
\right)
=\chi(G')=\ind(d^\alpha)
\]
for $\alpha>0$. The graph $G'$ is obtained from $G$ by removing all loop-edges.
\end{thm}

\begin{proof}
In both cases, each last equality is stated in Corollary \ref{ind=chi}.

\smallskip
In the bounded case, i.e.\ when $G$ is simple,  the operator $D^\alpha$ was identified with a Z\'u\~niga operator, and its spectrum and eigenfunctions were identified \cite[Cor.\ 5]{brad_heatMumf}.
Namely, the Kozyrev wavelets supported in $\Omega_V$, together with finitely many functions constant on $p$-adic discs making up the sets $U_v$, as explained in  the proof of \cite[Lem.\ 4]
{brad_heatMumf}, form an orthogonal eigenbasis of $L^2(\Omega_V)$.
In the case that a vertex $v$ does not have a loop-edge attached to it, the
Kozyrev wavelets supported in $U_v$ are still eigenfunctions of $D^\alpha$, cf.\ the proof of Lemma \ref{compactnessCriterion}. This does not touch on the part of $L^2(\Omega)$ of functions which are constant on  $p$-adic discs covering the sets $U_v$. Hence, those finitely many functions identified in the simple graph case remain eigenfunctions also in the non-simple case.

\smallskip
In the general case, the heat kernel for $D_\ell^\alpha$ exists and has the form:
\[
p_{D^\alpha_\ell,t}(x,y)
=\sum\limits_{\lambda\in\Spec(D^\alpha_\ell)}
e^{-\lambda t}f_\lambda^\ell(x)f_\lambda^\ell(y)
\]
where $f_\lambda^\ell$ is a normalised eigenfunction for $D_\ell^\alpha$.
Accordingly, the following holds true:
\begin{align}\label{traceDiff}
\trace(p_{D_\ell^\alpha,t})
-\trace(p_{\Delta_\rho^\alpha,t})=
\chi(G')+\sum_{\lambda\in\Spec_{\res}(D_\ell^\alpha)}
e^{-\lambda t}
\end{align}
where $\Spec_{\res}(D_\ell^\alpha)$ is the part of the spectrum of $D_\ell^\alpha$ not coming from the (weighted) graph Laplacian of $G$, i.e.\ not from Kozyrev wavelets and some of the other eigenfunctions of $D^\alpha$ described above.
This does make sense, because a Kozyrev wavelet on $\mathds{Q}_p$ takes precisely $p$ values, each on a distinct maximal non-trivial subdisc of its support, and by the construction of $D_\ell$ acting on a finite-dimensional part of $L^2(\Omega_V)$ as the operator $D^\alpha$.

\smallskip
Hence, it follows in the more general case of a vertex possibly not having a loop-edge that
\[
\lim\limits_{\ell\to\infty}\lim\limits_{t\to\infty}
\left(
\trace(p_{D_\ell^\alpha,t})
-\trace(p_{\Delta_\rho^\alpha,t})
\right)=\chi(G')
\]
for $\alpha>0$, as asserted.

\smallskip
In the case that $D^\alpha$ is compact, the identity
\[
\lim\limits_{\ell\to\infty} p_{D_\ell^\alpha,t}
=\sum\limits_{\lambda\in\Spec(D^\alpha)}e^{-\lambda t}f_\lambda(x)f_\lambda(y)
\]
does make sense where $\mathset{f_\lambda}$ is an orthonormal eigenbasis of $L^2(\Omega_V)$ 
for $D^\alpha$,
as the right-hand side converges due to the compactness of the operator. By taking limits in (\ref{traceDiff}), it follows that
\[
\trace(p_{D^\alpha,t})-\trace(p_{\Delta_\rho^\alpha})
=\chi(G')+\sum\limits_{\Spec_{\res}(D^\alpha)}e^{-\lambda t}
\]
where $\Spec(D^\alpha)_{\res}$ is defined in an analogous manner as in the more general case. It follows that 
\[
\lim\limits_{t\to\infty}\left(\trace(p_{D^\alpha,t})-\trace(p_{\Delta^\alpha,t})\right)=\chi(G')
\]
also in this case, as asserted. This proves the theorem.
\end{proof}

\section{Application to $p$-adic Mumford curves}

The $p$-adic index theorem (Theorem \ref{IndexTheorem}) has applications to Mumford curves, because of their property as a compact $p$-adic manifold with an inner graph structure. The theory of Mumford curves can be found in e.g.\ \cite{GvP1980,FP2004}, from which their only property needed is their relationship to finite graphs.
\newline

In the following subsections, different kinds of finite graphs are built from a good fundamental domain $F$ in the sense of \cite[Def.\ I.4.1.3]{GvP1980}, and then Theorem \ref{IndexTheorem} is applied in order to extract information from a Mumford curve via the heat kernels.

\subsection{The genus of a Mumford curve}

Let $F\subset\mathds{Q}_p$ be a good fundamental domain (in the sense of \cite[Def.\ I.4.1.3]{GvP1980})
of the action of the Schottky group giving rise to a Mumford curve $X$ of genus $g(X)>0$. Let $T(F)$ be the subtree of the Bruhat-Tits tree $\mathscr{T}_p$ of $\mathds{Q}_p$ associated with $F$, defined in \cite[\S 4.1]{brad_heatMumf} as the \emph{dendrogram} for the finitely many points, each taken from a distinct inner hole of $F$.
Then a so-called \emph{semi-stable reduction graph} of $X$ is given by attaching $g(X)$ edges to $T(F)$. This is in fact the quotient graph of \cite[Ch.\ I.3]{GvP1980}
obtained by the action of the Schottky group of $X$ on an infinite subtree of the Bruhat-Tits tree $\mathscr{T}_p$ of $\mathds{Q}_p$.
This is a simple graph which becomes a so-called \emph{stable} graph after contracting all edges attached to a vertex of degree $2$.
Let $G=(V,E(G))$ be the graph obtained from that simple graph by attaching to all vertices a loop-edge. 
Thus
\[
F=\bigcup\limits_{v\in V(T_F)}U_v
\]
where $U_v\subset\mathds{Q}_p$ is a holed disc, as explained in \cite[\S 4.1]{brad_heatMumf}.
The corresponding edge space is
\[
\Omega_E=\bigsqcup\limits_{e\in E(G)}U_{o(e)}\times U_{t(e)}
\]
and the 
results of the previous section can be applied.

\begin{Corollary}
It holds true that
\[
\ind(d^\alpha_G)=\lim\limits_{t\to\infty}
\left(
\trace(p_{D^\alpha,t})-\trace(p_{\Delta_\rho^\alpha})
\right)
=1-g(X)=\chi(X)
\]
where $\chi(X)$ is the Euler characteristic  of the Mumford curve $X$ w.r.t.\ its simplicial cohomology.
\end{Corollary}

\begin{proof}
This is an immediate consequence of Theorem \ref{IndexTheorem}, because the graph $G$ is connected, and $g(X)=b_1(G')$ by the theory of Mumford curves, cf.\ \cite[Thm.\ III.2.12.2]{GvP1980}.
\end{proof}

\subsection{Counting holes of a good fundamental domain}

The graph $G$ is now defined as the complete graph (including all loop-edges) on the set  $\mathscr{H}$ of inner holes of a good fundamental domain $F\subset\mathds{Q}_p$.
Assume that the number of holes is $N$. Then
\[
\Omega_V=\bigsqcup\limits_{H\in\mathscr{H}}H\subset\mathds{Q}_p
\]
is the set defined by the vertices of $G$, whereas 
\[
\Omega_E=\Omega_V\times\Omega_V
\]
is defined by the edges of $G$.

\begin{Corollary}
It holds true that
\[
\ind(d_G^\alpha)
=\lim\limits_{t\to\infty}
\left(
\trace(p_{D^\alpha,t})-\trace(p_{\Delta_\rho^\alpha,t})
\right)
=-\frac12N^2+\frac32N
\]
where $N$ is the number of holes of a good fundamental domain $F\subset\mathds{Q}_p$ for $X$.
\end{Corollary}

\begin{proof}
This is an immediate consequence of Theorem \ref{IndexTheorem}.
\end{proof}

\begin{remark}
Observe that the genus $g(X)$ can be extracted from the number of holes:
\[
\ind(d_G^\alpha)=-\frac12(2g-1)^2+\frac32(2g-1)
\]
as $N=2g-1$ holds true \cite[\S I.4]{GvP1980}.
\end{remark}

\subsection{The number of vertices of the reduction graph}

Take as $G$ now the complete graph on the set of vertices of $T(F)$ for a good fundamental domain $F\subset\mathds{Q}_p$ for a Mumford curve $X$. Again, assume that all loop-edges are present.

\begin{Corollary}
It holds true that
\[
\ind(d_G^\alpha)
=
\lim\limits_{t\to\infty}
\left(
\trace(p_{D^\alpha,t})-\trace(p_{\Delta^\alpha,t})
\right)
=-\frac12\absolute{V(T_F)}^2
+\frac32\absolute{V(T_F)}
\]
where $F\subset\mathds{Q}_p$ is a good fundamental domain of $X$.
\end{Corollary}

\begin{proof}
This follows immediately from Theorem \ref{IndexTheorem}.
\end{proof}

\begin{remark}
The number of vertices of $T_F$ 
is a statistic about the geometry of a Mumford curve. E.g.\ in the case $g=1$, this number 
equals the thickness
of the annulus $F$, i.e.\ the negative $p$-adic valuation of  the parameter $q\in \mathds{Z}_p$ of the corresponding $p$-adic  Tate curve. Consult e.g.\ \cite[Ch.\ 5.1]{FP2004} for more about the Tate elliptic curve.
\end{remark}

\section*{Acknowledgements}

Andrew and  John Bradley,
\'Angel Mor\'an Ledezma and David Weisbart are warmly thanked for fruitful discussions.
Wilson Z\'u\~{n}iga-Galindo is warmly thanked for posing such a nice problem.
This work is partially supported by the Deutsche Forschungsgemeinschaft under project number 469999674.

\bibliographystyle{plain}
\bibliography{biblio}

\end{document}